\documentclass[a4paper]{amsart}

\usepackage[T1]{fontenc}
\usepackage[utf8]{inputenc}
\usepackage[english]{babel}
\usepackage{lmodern}
\usepackage{exscale}
\usepackage[babel]{microtype}
\usepackage{amsmath, amssymb, mathtools, mathrsfs}

\usepackage{csquotes}
\usepackage[style=alphabetic, sorting=nyt, maxnames=10, maxalphanames=5, backend=bibtex]{biblatex}
\usepackage{tikz}
\usetikzlibrary{matrix,shapes,arrows,calc,3d,decorations,decorations.pathmorphing,through,cd}
\usepackage{todonotes}

\usepackage[numbered]{bookmark}
\usepackage{hyperref}
\usepackage{amsthm}
\hypersetup{
  colorlinks,
  urlcolor = {blue!80!black},
  linkcolor = {red!70!black},
  citecolor = {green!50!black}
}

\theoremstyle{plain}
  \newtheorem{thm}{Theorem}

  \newtheorem{lemma}[thm]{Lemma}
\theoremstyle{definition}

\usepackage{fullpage}
\usepackage{stackengine}
\usepackage{auxhook}
\usepackage{amsfonts}

\tikzset{ext/.style={circle, draw,inner sep=1pt},int/.style={circle,draw,fill,inner sep=1pt},nil/.style={inner sep=1pt}}
\tikzset{exte/.style={circle, draw,inner sep=3pt},inte/.style={circle,draw,fill,inner sep=3pt}}
\tikzset{diagram/.style={matrix of math nodes, row sep=3em, column sep=2.5em, text height=1.5ex, text depth=0.25ex}}
\tikzset{diagram2/.style={matrix of math nodes, row sep=0.5em, column sep=0.5em, text height=1.5ex, text depth=0.25ex}}

\newcommand{\co}[2]{\left[{#1},{#2}\right]} % commutator

\newcommand{\op}{\mathcal}

\newcommand{\bpm}{\begin{pmatrix}}
\newcommand{\epm}{\end{pmatrix}}

\DeclareMathOperator{\Hom}{Hom}

\DeclareMathOperator{\Def}{Def}
\DeclareMathOperator{\id}{id}

\newcommand{\Der}{\mathrm{Der}}

\newcommand{\beq}[1]{\begin{equation}\label{#1} 
}
\newcommand{\eeq}{\end{equation}}

\newcommand{\stadpole}{
\begin{tikzpicture}[baseline=-0.5ex]
\node[ext](v) at(0,0) {};
\draw (v) to [ out=120, in=60, loop] (v);
\end{tikzpicture}}

\newcommand{\dtadpole}{
\begin{tikzpicture}[baseline=-0.5ex]
\node[ext](v) at(0,0) {};
\draw[dashed] (v) to [out=-120, in=-60, loop] (v);
\end{tikzpicture}}

\newcommand{\sedge}{
\begin{tikzpicture}[baseline=-0.5ex]
\node[ext](v) at(0,0) {};
\node[ext](w) at(.5,0) {};
\draw (v) edge [out=45, in=135] (w);
\end{tikzpicture}}

\newcommand{\dedge}{
 \begin{tikzpicture}[baseline=-0.5ex]
\node[ext](v) at(0,0) {};
\node[ext](w) at(.5,0) {};
\draw (v) edge [out=-45, in=-135, dashed] (w);
\end{tikzpicture}}

\newcommand{\extnode}{
\begin{tikzpicture}[baseline=-0.5ex]
\node[ext](v) at(0,0) {};
\end{tikzpicture}}

\newcommand{\intext}{
\begin{tikzpicture}[baseline=-0.5ex]
\node[ext](v) at(0,-0.25) {};
\node[int](w) at(0,.25) {};
\draw (v) edge [] (w);
\end{tikzpicture}}

\newcommand{\intint}{
\begin{tikzpicture}[baseline=-0.5ex]
\node[int](v) at(0,-0.25) {};
\node[int](w) at(0,.25) {};
\draw (v) edge [] (w);
\end{tikzpicture}}

\begin{document}

\title{Homotopy derivations of the framed little discs operads}

\author[S. Brun]{Simon~Brun}
\address{Department of Mathematics, ETH Zurich, Zurich, Switzerland}
\email{simonbrun(at)math.ethz.ch}


\begin{abstract}

We study the homotopy derivations of the framed little discs operads, which correspond to the homotopy derivations of the $BV_{2n}$ operads. By extending a result by Willwacher about the homotopy derivations of the $e_n$ operads we show that the homotopy derivations of the $BV_{2n}$ operads may be described through the cohomology of a suitable graph complex. We will present an explicit quasi-isomorphic map.
\end{abstract}

\maketitle

\tableofcontents

\section{Introduction}
\label{sec:introduction}

The little n-discs operad $D_n$ is an operad in the category of topological spaces. Its arity k operations $D_n(k)$ correspond to the space of embeddings of $k$ copies of the unit n-disc to itself, $\prod_k D_n \rightarrow D_n$, in such a way that the embedding maps are a composition of translations and dilations. The framed little n-discs operad $fD_n$ allows additionally for rotations in the embeddings. Precisely, the n-discs operad $D_n$ has a left action of $SO(n)$ by rotating the little n-discs around their center. Then, the framed little n-discs operad $fD_n$ is a topological operad defined as semidirect product 
\[
fD_n = D_n \rtimes SO(n)
\]
with the operations in arity k given by
\[
fD_n(k) = D_n(k) \rtimes SO(n)^k
\]
as described in Definition 2.1 of \cite{SalvatoreWahl03}. Its operadic composition 
\[
\gamma_{fD_n}: (D_n \rtimes SO(n))(k)\times((D_n \rtimes SO(n))(n_1)\times\cdots\times(D_n \rtimes SO(n))(n_k))\rightarrow (D_n \rtimes SO(n))(n_1+\cdots + n_k)
\]
is defined as
\[
\gamma_{fD_n} ((a, \mathbf{g}),((b_1, \mathbf{h}^1),\dots,(b_k,\mathbf{h}^k))) := (\gamma_{D_n} (a, (g_1 b_1,\dots, g_k b_k)), g_1\mathbf{h}^1,\dots, g_k\mathbf{h}^k)
 \]
where $a,b_1,\dots, b_k\in D_n$ and $\mathbf{g},\mathbf{h}^1,\dots,\mathbf{h}^k \in SO(n)$.

While the homology of the little n-discs operad corresponds to the Gerstenhaber operad $e_n$
\[
H(D_{n}) \cong e_{n},
\]
the homology of the framed little 2n-discs operad is described by the $BV_{2n}$ operad, precisely,
\[
H(fD_{2n}) \cong BV_{2n} \rtimes H(SO(2n-1)),
\]
as stated in theorem 5.4 of \cite{SalvatoreWahl03}.

Batalin-Vilkovisky algebras, i.e. algebras over the $BV_2$ operad, have their origin in Physics and were mathematically introduced by Getzler in \cite{Getzler94}. 

Kontsevich \cite{Kontsevich1999} as well as Lambrechts and Volić \cite{LambrechtsVolic2014} showed that the little n-discs operad $D_n$ is formal. Ševera \cite{Severa10} as well as Giansiracusa and Salvatore \cite{GiansiracusaSalvatore10} extended this proof to the framed little 2-discs operad $fD_2$, i.e. there exists a zig-zag of quasi-isopmorhisms of of dg operads between the operad of rational singular chains of the framed little 2-discs operad and its homology, which is isomorphic to the $BV_2$ operad:
\[
C_{\bullet}(fD_2) \quad \tilde{\leftarrow} \quad . \quad \tilde{\rightarrow} \quad H_{\bullet}(fD_2) \cong BV_2,
\]
\cite[Theorem 10]{CarrilloTonksVallette12}.

Gálvez-Carrillo, Tonks and Vallette \cite{CarrilloTonksVallette12} extended the theory about Koszul dual operads of quadratic operads to inhomogeneous quadratic operads, as for example the $BV_{2n}$ operad. We will describe the Koszul dual operad $BV_{2n}^{!}$ in section \ref{bv_koszul}. The cobar construction of the Koszul dual cooperad $BV_{2n,\infty}:=\Omega BV_{2n}^{\text{!`}}$ provides quasi-free, but not minimal, resolution of the $BV_{2n}$ operad as stated in theorem 6 of \cite{CarrilloTonksVallette12}:
\[
BV_{2n,\infty}:=\Omega BV_{2n}^{\text{!`}}  \quad \tilde{\rightarrow} \quad BV_{2n}.
\]
Furthermore there is a quasi-isomorphism of dg operads 
\[
\Omega BV_{2}^{\text{!`}}  \quad \tilde{\rightarrow} \quad C_{\bullet}(fD_2)
\]
which lifts the resolution 
\[
BV_{2,\infty} \quad \tilde{\rightarrow} \quad BV_{2},
\]
\cite[Theorem 11 and 13]{CarrilloTonksVallette12}.

Hence, we can deduce important applications for homotopy $BV_2$ algebras, i.e. algebras over $BV_{2,\infty}$, as pointed out in corollary 12 and 14 of \cite{CarrilloTonksVallette12}: Any topological conformal field theory carries a homotopy $BV_2$ algebra structure. The same is true for the singular chain complex of the double loop space of a topological space endowed with an action of the circle. 

In \cite{Willwacher2015} Willwacher discribes the homotopy derivations of the $e_n$ operads governing n-algebras through the cohomology of Kontsevich's graph complex $GC_n$. Precisely, theorem 1.3 in \cite{Willwacher2015} states that 
\[
H(Der(e_{n,\infty})) \cong S^+\left(H(GC_{n,conn}^{\geq2})[-n-1]\oplus\mathbb{R}[-n-1]\right)[n+1].
\]
where $GC_{2n,conn}^{\geq2}$ denotes the connected graphs of Kontsevich's graph complex with at least bivalent vertices.

Expanding this result to the $BV_{2n}$ operads, we show that the homotopy derivations of the $BV_{2n}$ operads are quasi-isomorphic to the homology of a suitable graph complex. Precisely, in theorem \ref{quasiiso} and equation \ref{coho_hom_der_bv} we show, that 
\[
H(Der(BV_{2n,\infty})) \cong S^{+}_{\mathbb{R}[[u]]}\left(\left(H\left(GC_{2n,conn}^{\geq2}\right)[-2n-1] \oplus \mathbb{R}[-2n-1]\right)\left[[u]\right]\right)[2n+1],
\]
where $u$ is an even variable emerging in the Koszul dual operad $BV_{2n}^{!}$, since $BV_{2n}$ is an inhomogeneous quadratic operad. The differential in the considered graph complex has a first order contribution in $u$ additional to the vertex splitting differential inherent to Kontsevich's graph complex. 

Furthermore, we will present an explicit combinatorial map and prove that it is a quasi-isomorphism.

In \cite{Willwacher2015} Willwacher proved that the zeroth cohomology of the homotopy derivations of the $e_2$ operad is isomorphic to the Grothendieck-Teichm\"uller Lie algebra plus one class. As a corollary we extend this fact to the cohomology of the homotopy derivations of the $BV_2$ opeard in theorem \ref{zerocoho}:
\[
H^0(Der(BV_{2,\infty}))\cong  \mathfrak{grt} := \mathfrak{grt_1}\rtimes\mathbb{R}.
\]

\section{Preliminaries and basic notation}
\subsection{General notation}
In this paper we always work over the ground field $\mathbb{R}$. The degree of an element $x$ of a graded or differential graded (dg) vector space $V$ will be denoted by $|x|$ and the the r-fold desuspension by $V[r]$. For dg vector spaces we use cohomological convention, i.e. all differentials have degree one. Furthermore, we will always use a lexicographic ordering of odd objects of graded vector spaces, i.e. odd components of objects are ordered according to the appearance of the object in the formula from left to right.

We denote the completed symmetric product space of a vector space $V$ by
\[
S(V) = \mathbb{R} \oplus S^+(V) = \mathbb{R} \oplus \prod_{j\geq1}\left(V^{\otimes j}\right)^{\mathbb{S}_j}
\]
where the symmetric group $\mathbb{S}_n$ acts by permutations of the factors. 

We will consider the tensor coalgebra $S^+(V)$ equipped with the deconcatenation coproduct $\Delta: S^+(V)\rightarrow S^+(V)\otimes S^+(V)$ given by
\begin{equation}
\Delta (v_1\cdots v_n) := \sum_{i=1}^{n-1}v_1\cdots v_{i}\otimes v_{i+1}\cdots v_{n}.
\label{deconcatenation}
\end{equation}

Concerning $\mathbb{S}$-modules as well as operads we use the conventions from the textbook \cite{LodayVallette2012} by Loday and Vallette. 
We denote the $n$-ary operations of an operad $\op P$ by $\op P(n)$. The operadic $r$-fold desuspension $\op P\{r\}$ is an operad with 
\[
\op P \{r\}(n)=\op P(n) \otimes sgn_n^{\otimes r}\left[(n-1)r\right]
\]
where $sgn_n$ is the sign representation of the symmetric group $\mathbb{S}_n$.

As defined in section 5.1.4 of \cite{LodayVallette2012} the composition of two $\mathbb{S}$-modules $M$ and $N$ is given as 
\begin{equation}
M\circ N=\bigoplus_{k\geq0} M(k)\otimes_{\mathbb{S}_k}N^{\otimes k}.
\label{comp_smodules}
\end{equation}
Furthermore, corollary 5.1.4 of \cite{LodayVallette2012} states that in arity $n$ the composition is given by
\[
(M\circ N)(n)=\bigoplus_{k\geq0} M(k)\otimes_{\mathbb{S}_k}\left(\bigoplus Ind~^{\mathbb{S}_n}_{\mathbb{S}_{i_1}\times\cdots\times\mathbb{S}_{i_1}}\left(N(i_1)\otimes\cdots\otimes N(i_k)\right)\right)
\]
where the sum extends over all the nonegative $k$-tupels ${(i_1,\dots,i_k)}$ with $i_1+\cdots+i_k=n$. Hence, the space $(M\circ N)(n)$ is spanned by the equivalence classes of the elements 
\[
(\mu;\nu_1,\dots\nu_k;\sigma)
\]
where $\mu \in M(k), \nu_1\in N(i_1),\dots,\nu_k \in N(i_k),\sigma \in Sh(i_1,\dots,i_k)$, as described in section 5.1.7 of \cite{LodayVallette2012}.

Let $\op P$ be a symmetric operad with composition map
\[
\gamma: \op P \circ \op P \rightarrow \op P
\] 
constituted by the linear maps
\[
\gamma(i_1,\dots,i_k): \op P(k) \otimes \op P(i_1)\otimes \cdots \otimes \op P(i_k) \rightarrow \op P (i_1+\cdots + i_k).
\]
Furthermore, let $\mu \in \op P(m)$ and $\nu \in \op P(n)$ be two operations. The partial composition $(\mu,\nu)\mapsto \mu \circ_i \nu \in \op P (m-1+n)$ is defined by
\[
\mu \circ_i \nu := \gamma(\mu;id,\dots,,id,\nu,id,\dots,id).
\]

Additionally to the composition of $\mathbb{S}$-modules we use $\circ$ also to denote the composition of maps. It will be clear from the context, which notation is meant. 

\subsection{BV operad and its Koszul dual operad}
\label{bv_koszul}
In order to study the homotopy derivations of the $BV_{2n}$ operad we will consider its cofibrant resolution, obtained by the cobar construction for the Koszul dual cooperad according to the Koszul duality theory for inhomogeneous quadratic operads  introduced by Gálvez-Carrillo, Tonks and Vallette in \cite{CarrilloTonksVallette12}. Therefore, let us review the Kozul dual operad of the $BV_{2n}$ operad, as described in the sections 13.7.4 and 7.8.7 of \cite{LodayVallette2012}.

Let $\op P$ be an operad, we denote its Koszul dual cooperad by $\op P^{\text{!`}}$ and its Koszul dual operad by  $\op P^{!}$. The $\op BV_{2n}$ operad has a representation as an inhomogeneous quadratic operad with the three generators $m=\cdot\wedge\cdot$ and $c=[\cdot,\cdot]$ in arity two as well as $\Delta$ in arity one. $\mathbb{R}m$ is a trivial representation of $\mathbb{S}_2$ in degree $0$, $\mathbb{R}c$ is a trivial representation of $\mathbb{S}_2$ in degree $2n-1$, and  $\mathbb{R}\Delta$ is a one dimensional graded vector space in degree $2n-1$. 

As $\mathbb{S}$-module the $\op BV_{2n}$ operad is isomorphic to
\[
\op BV_{2n} \cong e_{2n}  \circ T(\Delta)/ (\Delta^2) \cong Com \circ \op Lie\{-2n+1\} \circ T(\Delta)/ (\Delta^2)
\]
where $\circ$ refers to the composition of $\mathbb{S}$-modules as defined in equation \ref{comp_smodules} and $e_{n}$ is the operad that governs $e_{n}$-algebras.

Since the $BV_{2n}$ operad is an inhomogeneous quadratic operad its koszul dual operad is a qudratic operad with a differential $d_{BV}$
\[
BV_{2n}^{!}=\left(qBV_{2n}^{!}, d_{BV}\right)
\]
where the quadratic operad is isomorphic as $\mathbb{S}$-module to
\[
qBV_{2n}^{!}\cong S(u) \circ  e_{2n}\{-2n\}\cong S(u)\circ \op Com\{-2n\} \circ \op Lie\{-1\}
\]
with $|u|=2n$.
Let $u^d L_1\wedge \dots \wedge L_N$ denote a homogeneous element of $BV_{2n}^{!}$ with $L_i$ being a $\op Lie$ word, then the $BV$ differential reads as
\begin{equation}
d_{BV}\left(u^d L_1\wedge \dots \wedge L_N\right) = \sum_{1\leq i<j\leq N} \pm u^{(d+1)} [L_i,L_j]\wedge L_1\wedge \dots \wedge \hat{L}_i \wedge \dots \wedge \hat{L}_j \wedge \dots \wedge L_N
\label{d_bv}
\end{equation}
Here, $\hat{L}_i$ indicates that $L_i$ is missing. 

\subsection{Convolution dg Lie algebra and deformation complex}
\label{conv_lie_alg}
The homotopy derivations of the $BV_{2n}$ operad will be identified with a convolution dg Lie algebra. Concerning convolution dg Lie algebras we refer to section 6.4 of \cite{LodayVallette2012}.

Let $\op C$ be a coaugmented cooperad with $\op C(1)$ one dimensional, $\op C(0)=0$ and $\op P$ an augmented operad, we denote the convolution dg Lie algebra by 
\[
\Hom_{\mathbb{S}}(\op C, \op P) = \prod_{N\geq 1} \Hom_{\mathbb{S}_N}(\op C(N), \op P(N)).
\] 
$\Omega(\op C)$ will denote the quasi free opearad obtained by the cobar construction. A homomorphism of dg operads $\alpha: \Omega(\op C)\to \op P$ determines a Maurer-Cartan element in the convolution dg Lie algebra, also denoted by $\alpha$. We twist by this Maurer-Cartan element to obtain a Lie algebra
\[
\Def(\Omega(\op C)\stackrel{\alpha}{\to}  \op P)
\] 
and call it deformation complex of the map $\alpha$.

The convolution dg Lie algebra $\Hom_{\mathbb{S}}(\op C, \op P)$ is isomorphic to 
\begin{equation}
\Hom_{\mathbb{S}}(\op C, \op P)\cong  \op C^* \hat\otimes_{\mathbb{S}} \op P := \prod_{N\geq 1} \op C^*(N)  \hat\otimes_{\mathbb{S_N}} \op P(N)
\label{convolutionLieAlg}
\end{equation}
where the completion of the tensor product is with respect to the cohomological filtration.

Let $f=f_1\otimes f_2$ and $g=g_1\otimes g_2$ be two elements of the convolution dg Lie algebra in the form (\ref{convolutionLieAlg}), then the pre-Lie bracket is given by
\[
\{f,g\} = \{f_1\otimes f_2, g_1\otimes g_2\} =\sum_i (-1)^{|f_2||g_1|}(f_1\circ_i g_1) \otimes (f_2\circ_i g_2) 
\]
and the corresponding Lie bracket reads as
\begin{equation}
[f,g] = \{f,g\} - (-1)^{|f||g|}\{g,f\}.
\label{lie_b_conv}
\end{equation}
Here, the composition $\circ_i$ refers to the partial operadic composition of the operad $\op C^*$ and $\op P$ respectively.

\subsection{Graph complexes}
We will consider different kinds of graph complexes. First we consider M. Kontsevich's graph complex $GC_{2n}$ as defined in section 3 of \cite{Willwacher2015}. Elements of this graded vector space are depicted as undirected graphs with one kind of unlabelled black vertices of valence greater or equal three and with no tadpols. Edges have an odd degree of $1-2n$, i.e. one needs to choose an ordering of the edges. In order to obtain an element of $GC_{2n}$ one has to sum over all possible ways of assigning labels to the vertices and divide by the order of the symmetry group of the graph. The vertices of graphs in $GC_{2n}$ have to be at least trivalent. If we also allow bivalent vertices we denote the corresponding graph complex by $GC_{2n}^{\geq 2}$. Furthermore, we consider only graphs without tadpoles, loops with one vertex, as depicted in the following picture.
\[
\begin{tikzpicture}[scale=1,
helper/.style={outer sep=0,inner sep=0,minimum size=5,shape=coordinate},
default edge/.style={draw},
 every loop/.style={out=140, in=50, looseness=.8, distance=.8cm }]
 
\node (v0) at (0.600000023841858,0.799999952316284) [int] {};
\node (v2) at (0.600000023841858,0.399999976158142) [helper] {};
\node (v3) at (0.400000005960464,0.399999976158142) [helper] {};
\node (v4) at (0.800000011920929,0.399999976158142) [helper] {};

\draw[default edge] (v0)--(v3);
\draw[default edge] (v0)--(v2);
\draw[default edge] (v0)--(v4);
\draw (v0) to [out=70,in=110,loop] ();
\end{tikzpicture}
\]

We define the operation $\Gamma_1 \bullet \Gamma_2$, which means that we sum over all vertices of $\Gamma_1$, each time insert $\Gamma_2$ in the chosen vertex of $\Gamma_1$ and sum over all possible ways to reconnect the incident edges of the chosen vertex of $\Gamma_1$ to one of the vertices of $\Gamma_2$. 

Let $\Gamma$ be a graph in $GC_{2n}$. We introduce the following notation. 
\begin{equation}
\begin{tikzpicture}[baseline=-0.65ex]
\node[draw, dashed, circle] (G) at (0,0) {$\Gamma$};
\node[](v4) at(0,-1) {};
\draw (G) edge [] (v4);
\end{tikzpicture}\\
\label{add_edge}
\end{equation}
We sum over all vertices of $\Gamma$ and connect for each term the incident edge to the corresponding vertex. If there are several incident edges, we sum over all possible ways to connect the incident edges to vertices of $\Gamma$.
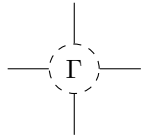
\[
\begin{tikzpicture}[baseline=-0.65ex]
\node[draw, dashed, circle] (G) at (0,0) {$\Gamma$};
\node[](v1) at(1,0) {};
\draw (G) edge [] (v1);
\node[](v2) at(0,1) {};
\draw (G) edge [] (v2);
\node[](v3) at(-1,0) {};
\draw (G) edge [] (v3);
\node[](v4) at(0,-1) {};
\draw (G) edge [] (v4);
\end{tikzpicture}\\
\]

The following picture shows an example.

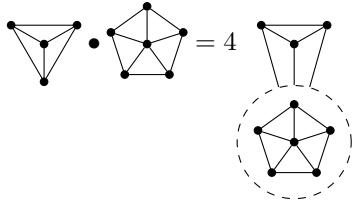
\[
\begin{tikzpicture}[baseline=-0.5ex]
\draw (30:0.5) node [int] (v18) {} -- (150:0.5) node [int] (v21) {} -- (-90:0.5) node [int] (v19) {} -- (0:0) node [int] (v20) {} -- (v18) -- (v19);
\draw  (v20) edge (v21);
\end{tikzpicture}
\bullet
\begin{tikzpicture}[baseline=-0.5ex]
\draw (90:0.5) node [int] (v22) {} -- (162:0.5) node [int] (v26) {} -- (-126:0.5) node [int] (v24) {} -- (-54:0.5) node [int] (v1) {} -- (18:0.5) node [int] (v23) {} -- (0:0) node [int] (v25) {} -- (v22) -- (v23);
\draw (v24) -- (v25) -- (v26);
\draw (v1) -- (v25);
\end{tikzpicture}
= 4
\begin{tikzpicture}[baseline=-0.5ex]
\node[draw, dashed, circle, minimum size=1.5cm] (G) at (-90:1.3) {};
\draw (30:0.5) node [int] (v18) {} -- (150:0.5) node [int] (v21) {} -- (G) -- (0:0) node [int] (v20) {} -- (v18) -- (G);
\draw  (v20) -- (v21);
\begin{scope}[shift={(0,-1.3)}]
\draw (90:0.5) node [int] (v22) {} -- (162:0.5) node [int] (v26) {} -- (-126:0.5) node [int] (v24) {} -- (-54:0.5) node [int] (v1) {} -- (18:0.5) node [int] (v23) {} -- (0:0) node [int] (v25) {} -- (v22) -- (v23);
\draw (v24) -- (v25) -- (v26);
\draw (v1) -- (v25);
\end{scope}
\end{tikzpicture}
\]

The Lie bracket on $GC_{2n}$ can then be described in the following way
\[
[\Gamma_1,\Gamma_2]=\Gamma_1\bullet\Gamma_2 - (-1)^{|\Gamma_1||\Gamma_2|}\Gamma_2\bullet\Gamma_1.
\]
Using this notation, the differential on the graph complex $GC_{2n}$ is defined as
\begin{equation}
\tilde{\delta}\Gamma = \Big[\intint, \Gamma \Big] =  \intint  \bullet \Gamma - (-1)^{|\Gamma |}  \Gamma  \bullet  \intint.
\label{delta_tilde}
\end{equation}

Note that the differential $\tilde{\delta}$ does neither split nor connect two connected components, the graph complex can be written as completed symmetric product of connected graphs
\[
GC_{2n} = S^+\left(GC_{2n,conn}\right).
\]

The cohomological degree of a connected component corresponds to 2n (\#vertices - 1) - (2n - 1) (\#edges).

Furthermore, we consider the operad $Graphs_{2n}$ as described in definition 3.6 of \cite{Willwacher2015}. Generators of this graded vector space are depicted by graphs with two kinds of vertices, external vertices which are numbered and depicted as empty (white) dots as well as internal vertices which are depicted as black dots and are indistinguishable and thus unnumbered. 

We consider only graphs without tadpoles, a loop with one vertex, at internal and external nodes as depicted in the following picture.
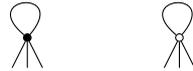
\[
\begin{tikzpicture}[scale=1,
helper/.style={outer sep=0,inner sep=0,minimum size=5,shape=coordinate},
default edge/.style={draw},
 every loop/.style={out=140, in=50, looseness=.8, distance=.8cm }]
 
\begin{scope}[shift={(-1,0)}]
\node (v0) at (0.600000023841858,0.799999952316284) [int] {};
\node (v2) at (0.600000023841858,0.399999976158142) [helper] {};
\node (v3) at (0.400000005960464,0.399999976158142) [helper] {};
\node (v4) at (0.800000011920929,0.399999976158142) [helper] {};

\draw[default edge] (v0)--(v3);
\draw[default edge] (v0)--(v2);
\draw[default edge] (v0)--(v4);
\draw (v0) to [out=70,in=110,loop] ();
\end{scope}

\begin{scope}[shift={(1,0)}]

\node (v0) at (0.600000023841858,0.799999952316284) [ext] {};
\node (v2) at (0.600000023841858,0.399999976158142) [helper] {};
\node (v3) at (0.400000005960464,0.399999976158142) [helper] {};
\node (v4) at (0.800000011920929,0.399999976158142) [helper] {};

\draw[default edge] (v0)--(v3);
\draw[default edge] (v0)--(v2);
\draw[default edge] (v0)--(v4);
\draw (v0) to [out=70,in=110,loop] ();
\end{scope}
\end{tikzpicture}
\]

Furthermore, we consider only graphs with all internal vertices at least trivalent and with no connected components consisting entirely of internal vertices. Edges have an odd degree of $1-2n$, i.e. one needs to choose an ordering of the edges. 

An operation in arity N has N numbered external vertices. The following picture shows connected graphs with 3 and 2 external vertices respectively.
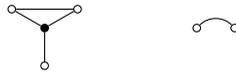
\[
\begin{tikzpicture}[]
\begin{scope}[shift={(-1,0)}]
\draw (30:0.5) node [ext] (v1) {};
\draw (150:0.5) node [ext] (v2) {};
\draw (-90:0.5) node [ext] (v3) {};
\draw (0:0) node [int] (v0) {};
\draw  (v0) edge (v1);
\draw  (v0) edge (v2);
\draw  (v0) edge (v3);
\draw  (v1) edge (v2);
\end{scope}
\begin{scope}[shift={(1,0)}]
\node[ext](v) at(0,0) {};
\node[ext](w) at(.5,0) {};
\draw (v) edge [out=45, in=135] (w);
\end{scope}
\end{tikzpicture}
\]

Let $\Gamma$ be a graph in $Graphs_{2n}$. We introduce again the following notation. 
\begin{equation}
\begin{tikzpicture}[baseline=-0.65ex]
\node[draw, dashed, circle] (G) at (0,0) {$\Gamma$};
\node[](v4) at(0,-1) {};
\draw (G) edge [] (v4);
\end{tikzpicture}
\label{add_edge_2}
\end{equation}
This time we sum over all internal and external vertices of $\Gamma$ and connect for each term the incident edge to the corresponding vertex. 

The partial operadic composition $\Gamma_1 \circ_i\Gamma_2$ corresponds to insert $\Gamma_2$ at the external vertex $i$ of $\Gamma_1$ and sum over all possible ways to reconnect the incident edges of vertex $i$ of $\Gamma_1$ to one of the vertices (internal or external) of $\Gamma_2$. The following picture shows an example of a partial operadic composition.

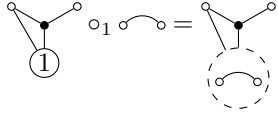
\[
\begin{tikzpicture}[baseline=-0.5ex]
\draw (30:0.5) node [ext] (v1) {};
\draw (150:0.5) node [ext] (v2) {};
\draw (-90:0.5) node [ext] (v3) {1};
\draw (0:0) node [int] (v0) {};
\draw  (v0) edge (v1);
\draw  (v0) edge (v2);
\draw  (v0) edge (v3);
\draw  (v2) edge (v3);
\end{tikzpicture}
\circ_1
\sedge
=
\begin{tikzpicture}[baseline=-0.5ex]
\draw (30:0.5) node [ext] (v1) {};
\draw (150:0.5) node [ext] (v2) {};
\node[draw, dashed, circle, minimum size=0.8cm] (v3) at (-90:0.7) {};
\draw (0:0) node [int] (v0) {};
\draw  (v0) edge (v1);
\draw  (v0) edge (v2);
\draw  (v0) edge (v3);
\draw  (v2) edge (v3);
\begin{scope}[shift={(-0.25,-0.75)}]
\node[ext](v) at(0,0) {};
\node[ext](w) at(.5,0) {};
\draw (v) edge [out=45, in=135] (w);
\end{scope}
\end{tikzpicture}
\]

Let $\Gamma_1$ be an operation in arity $N_1$ and $\Gamma_2$ be an operation in arity $N_2$ then the Lie bracket on $Graphs_{2n}$ induced by the operadic composition is defined as
\[
[\Gamma_1,\Gamma_2]=\sum_{i \in N_1}\Gamma_1\circ_i\Gamma_2 - (-1)^{|\Gamma_1||\Gamma_2|}\sum_{i \in N_2}\Gamma_2\circ_i\Gamma_1.
\]
The differential on the graph complex $Graphs_{2n}$ is defined in the following way
\begin{equation}
\delta\Gamma = \Big[\intext,\Gamma\Big] -(-1)^{|\Gamma|}\frac{1}{2} \Gamma \bullet\intint = \intext\circ_1\Gamma + (-1)^{|\Gamma|}\sum_{i \in \text{int. vert. of }\Gamma}\Gamma\circ_i\intext-(-1)^{|\Gamma|}\frac{1}{2} \Gamma \bullet\intint.
\label{delta}
\end{equation}

Finally, we consider the operad $BVGraphs_{2n}$. It differs from the operad $Graphs_{2n}$ by allowing tadpoles at external vertices and is isomorphic as $\mathbb{S}$-module to 
\[
BVGraphs_{2n} \cong Graphs_{2n} \circ S(\Delta)/(\Delta^2),
\]
where the generator $\Delta$ in arity one has degree $1-2n$ and corresponds to a tadpole at an external vertex. The differential on $BVGraphs_{2n}$ is the same as the one on $Graphs_{2n}$. Since it doesn't alter tadpoles the isomorphism is one of dg $\mathbb{S}$-modules
\begin{equation}
(BVGraphs_{2n},\delta) \cong (Graphs_{2n},\delta) \circ S(\Delta)/(\Delta^2).
\label{delta_no_tadpoles}
\end{equation}
For further details to the considered graph complexes we refer the reader to \cite{Willwacher2015}.

Let us consider a graph $\Gamma \in GC_{2n}$ with one external vertex attached to it

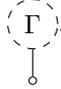
\[
\begin{tikzpicture}[baseline=-0.65ex]
\node[draw, dashed, circle] (G) at (0,0) {$\Gamma$};
\node[ext](v4) at(0,-0.75) {};
\draw (G) edge [] (v4);
\end{tikzpicture}.
\]
The differential $\delta$ does not have any influence on the external vertex and its incident edge, i.e.
\begin{equation}
\delta
\begin{tikzpicture}[baseline=-0.5ex]
\node[draw, dashed, circle, minimum size=0.5cm] (G) at (0,0) {$\Gamma$};
\node[ext](v4) at(0,-0.75) {};
\draw (G) edge [] (v4);
\end{tikzpicture}
= 
\begin{tikzpicture}[baseline=-0.5ex]
\node[draw, dashed, circle, minimum size=0.5cm] (G) at (0,0) {$\tilde{\delta}\Gamma$};
\node[ext](v4) at(0,-0.75) {};
\draw (G) edge [] (v4);
\end{tikzpicture}
\label{diff_not_affect_hair}
\end{equation}
where $\tilde{\delta}$ refers to the differential on $GC_{2n}$ as defined in equation \ref{delta_tilde}.

\section{Homotopy derivations of the $BV_{2n}$ operads}

\subsection{Homotopy derivations as deformation complex}

We identify the complex of homotopy derivations of the $BV_{2n}$ operad with the deformation complex
\begin{equation}
\Der(BV_{2n,\infty}) := \prod_{N\geq 1} \Hom_{S_N}\left(BV_{2n}^{\text{!`}}(N),  BV_{2n,\infty}(N)\right) [1] \cong \Def(BV_{2n,\infty}\stackrel{\id}{\to}  BV_{2n,\infty})[1]
\label{der_def}
\end{equation}
where $BV_{2n,\infty} = \Omega(BV_{2n}^{\text{!`}})$ denotes the quasi-free resolution of the $ BV_{2n}$ operad obtained by the cobar construction of the Koszul dual cooperad $BV_{2n}^{\text{!`}}$ as introduced by Gálvez-Carrillo, Tonks and Vallette in \cite{CarrilloTonksVallette12} and described in chapter 7.8.7 of \cite{LodayVallette2012}.
Since the projection $p: \Omega(BV_{2n}^{\text{!`}}) \twoheadrightarrow BV_{2n}$ is a quasi-isomorphism, the deformation complex is quasi-isomorphic to 
\[
\Def( BV_{2n,\infty}\stackrel{p}{\to} BV_{2n})[1].
\]
There is an injective map of operads $i: BV_{2n} \hookrightarrow BVGraphs_{2n}$ which maps the generating operations to graphs in the following way
\begin{align}
\cdot \wedge \cdot &\mapsto
\begin{tikzpicture}[baseline=-0.65ex]
\node[ext] at(0,0) {};
\node[ext] at(.5,0) {};
\end{tikzpicture}
&
\co{\cdot}{\cdot} &\mapsto\sedge
&
\Delta &\mapsto \frac{1}{2}\stadpole.
\end{align}
As proven by Kontsevich \cite{Kontsevich1999} as well as Lambrechts and Voli\'c \cite{LambrechtsVolic2014} the restriction of this map to $e_{2n} \hookrightarrow Graphs_{2n}$ is a quasi-isomorphism, which is also stated in proposition 3.9 in \cite{Willwacher2015}. From this we deduce that the Cohomology of the $BVGraphs_{2n}$ operad corresponds to the $BV_{2n}$ operad due to equation \ref{delta_no_tadpoles} and the operadic Künneth formula
\[
H^{\bullet}(BVGraphs_{2n}) \cong H^{\bullet}(Graphs_{2n}) \circ S(\Delta)/(\Delta^2) \cong e_{2n} \circ S(\Delta)/(\Delta^2) \cong BV_{2n}.
\]
Hence, the inclusion $BV_{2n} \hookrightarrow BVGraphs_{2n}$ is a quasi-isomorphism. Therefore, the considered deformation complex is quasi-isomorphic to 
\begin{equation}
\Def( BV_{2n,\infty}\stackrel{\alpha}{\to} BVGraphs_{2n})[1],
\label{def}
\end{equation}
where $\alpha = i \circ p: \Omega(BV_{2n}^{\text{!`}})\rightarrow BVGraphs_{2n}$.
Following the notation of section \ref{conv_lie_alg} this deformation complex can be written as
\[
\prod_{N\geq 1} \Hom_{S_N}\left(BV_{2n}^{\text{!`}}(N),  BVGraphs_{2n}(N)\right) [1] \cong \prod_{N\geq 1}BV_{2n}^{\text{!}}(N)\hat{\otimes}_{S_N} BVGraphs_{2n}(N)
\]
where $BV_{2n}^{\text{!}}$ denotes the Koszul dual operad of the $BV_{2n}$ operad as described in section \ref{bv_koszul}. 
Finally, the considered deformation complex \ref{der_def} is quasi-isomorphic to 
\begin{equation}
Def \coloneqq\prod_{N\geq 1}e_{2n}\{-2n\}(N)\hat{\otimes}_{S_N} BVGraphs_{2n}(N)\left[[u]\right].
\label{def_conv}
\end{equation}

Elements of this deformation complex can be visualised as graphs in the following way. We interpret homogeneous elements of the deformation complex as a direct sum of graphs with two kinds of edges and additionally with a power series in $u$. External vertices are arranged on a horizontal dotted line. The $BVGraphs$ part including internal vertices is depicted above the line using solid lines for edges. Furthermore, the map of operads  $e_{2n}\{-2n\} \hookrightarrow Graphs_{2n}$ which maps the generating operations as follows
\begin{align}
\cdot \wedge \cdot &\mapsto
\begin{tikzpicture}[baseline=-0.5ex]
\node[ext] at(0,0) {};
\node[ext] at(.5,0) {};
\end{tikzpicture}
&
\co{\cdot}{\cdot} &\mapsto\dedge
\end{align}
allows to depict the $e_{2n}$ part also as part of the graph using dashed edges connecting the external nodes below the horizontal line. Note that for example a nested double Lie $[\cdot,[\cdot,\cdot]]$ bracket in the $e_{2n}$ part will be mapped to
\[
\begin{tikzpicture}[baseline=-0.65ex]
\draw[dotted] (-1.5,0) -- (1.5,0);
\node[ext](v0) at(-1,0) {};
\node[ext](v1) at(0,0) {};
\node[ext](v2) at(1,0) {};
\draw (v0) edge [out=-45, in=-135, dashed]  (v1);
\draw (v1) edge [out=-45, in=-135, dashed]  (v2);
\end{tikzpicture}
\oplus
\begin{tikzpicture}[baseline=-0.65ex]
\draw[dotted] (-1.5,0) -- (1.5,0);
\node[ext](v0) at(-1,0) {};
\node[ext](v1) at(0,0) {};
\node[ext](v2) at(1,0) {};
\draw (v1) edge [out=-45, in=-135, dashed]  (v2);
\draw (v0) edge [out=-45, in=-135, dashed]  (v2);
\end{tikzpicture}.
\]

The following picture shows an example of an element of the deformation complex.
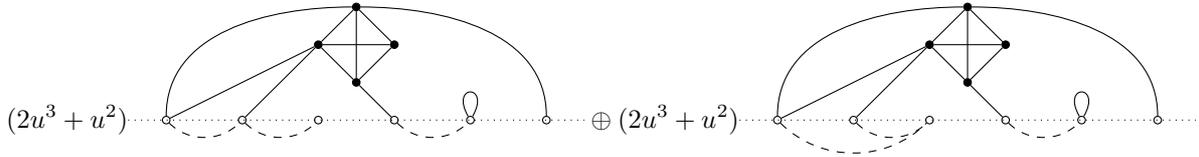
\[
(2u^3+ u^2)\begin{tikzpicture}[baseline=-0.65ex]
\node[int](i1) at(0,1.5) {};
\node[int](i2) at(-0.5,1) {};
\node[int](i3) at(0.5,1) {};
\node[int](i4) at(0,0.5) {};
\draw (i1) edge []  (i2);
\draw (i1) edge []  (i3);
\draw (i1) edge []  (i4);
\draw (i2) edge []  (i3);
\draw (i2) edge []  (i4);
\draw (i3) edge []  (i4);
\draw[dotted] (-3,0) -- (3,0);
\node[ext](v0) at(-2.5,0) {};
\node[ext](v1) at(-1.5,0) {};
\node[ext](v2) at(-0.5,0) {};
\node[ext](v3) at(0.5,0) {};
\node[ext](v4) at(1.5,0) {};
\node[ext](v5) at(2.5,0) {};
\draw (i1) edge [out=180, in=90]  (v0);
\draw (i2) edge []  (v0);
\draw (i2) edge []  (v1);
\draw (i4) edge []  (v3);
\draw (v4) to [out=120, in=60, loop]  (v4);
\draw (i1) edge [out=0, in=90]  (v5);
\draw (v0) edge [out=-45, in=-135, dashed]  (v1);
\draw (v1) edge [out=-45, in=-135, dashed]  (v2);
\draw (v3) edge [out=-45, in=-135, dashed]  (v4);
\end{tikzpicture}
\oplus
(2u^3+ u^2)
\begin{tikzpicture}[baseline=-0.65ex]
\node[int](i1) at(0,1.5) {};
\node[int](i2) at(-0.5,1) {};
\node[int](i3) at(0.5,1) {};
\node[int](i4) at(0,0.5) {};
\draw (i1) edge []  (i2);
\draw (i1) edge []  (i3);
\draw (i1) edge []  (i4);
\draw (i2) edge []  (i3);
\draw (i2) edge []  (i4);
\draw (i3) edge []  (i4);
\draw[dotted] (-3,0) -- (3,0);
\node[ext](v0) at(-2.5,0) {};
\node[ext](v1) at(-1.5,0) {};
\node[ext](v2) at(-0.5,0) {};
\node[ext](v3) at(0.5,0) {};
\node[ext](v4) at(1.5,0) {};
\node[ext](v5) at(2.5,0) {};
\draw (i1) edge [out=180, in=90]  (v0);
\draw (i2) edge []  (v0);
\draw (i2) edge []  (v1);
\draw (i4) edge []  (v3);
\draw (v4) to [out=120, in=60, loop]  (v4);
\draw (i1) edge [out=0, in=90]  (v5);
\draw (v0) edge [out=-45, in=-135, dashed]  (v2);
\draw (v1) edge [out=-45, in=-135, dashed]  (v2);
\draw (v3) edge [out=-45, in=-135, dashed]  (v4);
\end{tikzpicture}
\]

We will now express the differential on the considered deformation complex \ref{def} in the form of the convolution dg Lie algebra \ref{def_conv} using the Lie bracket defined in equation \ref{lie_b_conv}. The differential consists of different parts. The first part of the differential originates from the Koszul dual $BV_{2n}^{!}$ of the $BV_{2n}$ operad as described in equation \ref{d_bv}. Let $\Gamma$ denote an element of the deformation complex \ref{def} in the form of the convolution dg Lie algebra \ref{def_conv}. The the differential $d_{BV}$ can be written as
\[
d_{BV} \Gamma = u\frac{1}{2}\Big[\dtadpole, \Gamma \Big].
\]
In the same notation the contributions to the differential due to the twisting by the Maurer-Cartan element $\alpha$ in the deformation complex \ref{def} read as
\begin{align*}
d_{\wedge}\Gamma &=\Big[\dedge,\Gamma\Big]\\
d_{\co{}{}}\Gamma &=\Big[\sedge,\Gamma\Big]\\
d_{\Delta}\Gamma &=u\frac{1}{2}\Big[\stadpole,\Gamma\Big].
\end{align*}
Finally, there is the contribution form the differential $\delta$ inherent to the operad $BVGraphs_{2n}$ as defined in equation \ref{delta}. Summarising, the total differential on the considered deformation complex \ref{def} in the form of the convolution dg Lie algebra $Def$ as defined in \ref{def_conv} is given by
\begin{equation}
d = \delta + 
\Big[\dedge,\cdot \Big]
+
\Big[\sedge,\cdot\Big]
+ u\frac{1}{2}\left(
\Big[\stadpole,\cdot\Big]
+
\Big[\dtadpole, \cdot \Big]
\right).
\label{total_differential}
\end{equation}
Hence, the homotopy derivations $\Der(BV_{2n,\infty})$ of the $BV_{2n}$ operad are quasi-isomorphic to the convolution dg Lie algebra $(Def,d)$ in the form of equation \ref{def_conv} and with differential \ref{total_differential}.

\subsection{A map of coalgebras}
We can express the considered convolution dg Lie algebra $Def$ as completed symmetric product of its connected components
\begin{equation}
Def = S^{+}\left(\left(e_{2n}\{-2n\}\hat{\otimes}_{S} BVGraphs_{2n}\right)_{conn}\right)\left[[u]\right],
\label{def_s+}
\end{equation}
where connectedness refers to graphs connected via solid or dashed edges. Similarly the graph complex $GC_{2n}^{\geq2}$ splits into the completed symmetric product of its connected components
\[
GC_{2n}^{\geq2} = S^{+}\left(GC_{2n,conn}^{\geq2} \right).
\]

In the following we define a map $F$ of the cocommutative coalgebras
\begin{equation}
S^{+}\left(GC_{2n,conn}^{\geq2} \oplus \sedge\right)\left[[u]\right]
\xrightarrow[F]{}
S^{+}\left(\left(e_{2n}\{-2n\}\hat{\otimes}_{S} BVGraphs_{2n}\right)_{conn}\right)\left[[u]\right].
\label{map_f}
\end{equation}
Here, the coalgebra structure corresponds to the tensor coalgebra equipped with the deconcatenation coproduct defined in equation \ref{deconcatenation}. For both coalgebras the space of cogenerators is composed of the respective connected graphs and will be denoted by
\begin{equation}
\begin{aligned}
U&:=\left(GC_{2n,conn}^{\geq2} \oplus \sedge\right)\\
V&:=\left(e_{2n}\{-2n\}\hat{\otimes}_{S} BVGraphs_{2n}\right)_{conn}.
\end{aligned}
\label{cogenerators}
\end{equation}

The map $F$ of cocommutative coalgebras is defined via its projections on the cogenerators
\[
F_n: S^n(U)\rightarrow V.
\]
Let $\Gamma \in GC_{2n,conn}^{\geq2}$, the map $F_1: U \rightarrow V, \Gamma \mapsto \hat{\Gamma}$ corresponds to "adding a hair". In this context a hair consists of one external vertex connected to the graph by a solid edge . The image of the map $F_1$ equals the sum of graphs with one hair obtained by adding a hair to $\Gamma$ in all possible ways, i.e. at every vertex. Using the notation of equation \ref{add_edge} we can summarise the map $F_1$ as
\begin{equation}
\begin{aligned}
\Gamma &\mapsto \hat{\Gamma}=
\begin{tikzpicture}[baseline=-0.65ex]
\node[draw, dashed, circle] (G) at (0,0) {$\Gamma$};
\node[ext](v4) at(0,-0.75) {};
\draw (G) edge [] (v4);
\end{tikzpicture}\\
\sedge &\mapsto \sedge.
\end{aligned}
\label{f1}
\end{equation}

Let $\Gamma^1, \Gamma^2 \in V$ be connected graphs with one and two external vertices respectively. Furthermore let $\Gamma_1, \Gamma_2 \in V$ be two connected graphs. We consider the composition 
\[
\Gamma^1\circ_1\Gamma_1
\]
which is defined by inserting $\Gamma_1$ into the external vertex of $\Gamma^1$ and sum over all possible ways to connect the incident edges of $\Gamma^1$ to the vertices of $\Gamma_1$. Similarly the composition 
\[
\Gamma^2\circ(\Gamma_1,\Gamma_2) 
\]
is defined by inserting $\Gamma_1$ into the first and $\Gamma_2$ into the second vertex of $\Gamma^2$ and sum over all possible ways to connect the incident edges of $\Gamma^2$ to vertices of $\Gamma_1$ and $\Gamma_2$ respectively.

We will use the following notation. For connected graphs $\Gamma_i \in U$ we abbreviate
\[
\Gamma_1\dots\Gamma_n := \frac{1}{n!}\sum_{\sigma \in S_n}\Gamma_{\sigma(1)}\otimes\dots\otimes\Gamma_{\sigma(n)} \in S^n(U).
\]
Furthermore, let $\Gamma_i \in V$ for $i\neq 0$ be connected graphs with one external vertex and let $\Gamma_0 \in S^n(V)$ for some $n\geq1$, then we denote a symmetrised stack of compositions by
\[
\Gamma_1(\dots(\Gamma_n(\Gamma_0)))
:= \sum_{\sigma \in S_n}\Gamma_{\sigma(1)}\circ_1(\Gamma_{\sigma(2)}\circ_1(\dots\circ_1(\Gamma_{\sigma(n)}\circ_1\Gamma_0))).
\]

Since
\[
\sedge \, \sedge \, \Gamma_1\dots\Gamma_{n-2} =0
\]
vanishes due to the odd symmetry of interchanging the first two graphs we only need to consider at most one graph $\sedge$.

Using the introduced notation as well as the notation for connecting an edge to a graph as stated in equation \ref{add_edge_2} the projection $F_n$ is defined as
\begin{align*}
F_n(\Gamma_1\dots\Gamma_n)& :=u^{n-1}\hat{\Gamma}_1(\dots(\hat{\Gamma}_n(\extnode)))= u^{n-1}
\begin{tikzpicture}[baseline=8ex]
\node[ext] (v1) at (0,0) {};
\node[draw, circle] (v2) at (0,0.5) {};
\draw (v1) edge (v2);
\node[draw, dashed, circle, minimum size=0.9cm] (v3) at (0,0.3) {};
\node[draw, circle] (v4) at (0,1.25) {};
\draw (v3) edge (v4);
\node[draw, dashed, circle, minimum size=1.7cm] (v5) at (0,0.65) {};
\node[draw, circle] (v6) at (0,2) {};
\draw (v5) edge (v6);
\node[draw, dashed, circle, minimum size=2.5cm] (v7) at (0,1) {};
\node[draw, circle, dotted] (v8) at (0,2.75) {};
\draw (v7) edge[dotted] (v8);
\end{tikzpicture}\\
F_n(\sedge \Gamma_1\dots\Gamma_{n-1})& := u^{n-1} \left(\hat{\Gamma}_1(\dots(\hat{\Gamma}_{n-1}(\sedge)))+u\hat{\Gamma}_1(\dots(\hat{\Gamma}_{n-1}(\stadpole)))\right)\\
&=u^{n-1}
\begin{tikzpicture}[baseline=8ex]
\node[ext](v) at(-0.25,0.3) {};
\node[ext](w) at(.25,0.3) {};
\draw (v) edge [out=45, in=135] (w);
\node[draw, dashed, circle, minimum size=0.9cm] (v3) at (0,0.3) {};
\node[draw, circle] (v4) at (0,1.25) {};
\draw (v3) edge (v4);
\node[draw, dashed, circle, minimum size=1.7cm] (v5) at (0,0.65) {};
\node[draw, circle] (v6) at (0,2) {};
\draw (v5) edge (v6);
\node[draw, dashed, circle, minimum size=2.5cm] (v7) at (0,1) {};
\node[draw, circle, dotted] (v8) at (0,2.75) {};
\draw (v7) edge[dotted] (v8);
\end{tikzpicture}
+u^n
\begin{tikzpicture}[baseline=8ex]
\node[ext](v) at(0,0.2) {};
\draw (v) to [out=120, in=60, loop] (v);
\node[draw, dashed, circle, minimum size=0.9cm] (v3) at (0,0.3) {};
\node[draw, circle] (v4) at (0,1.25) {};
\draw (v3) edge (v4);
\node[draw, dashed, circle, minimum size=1.7cm] (v5) at (0,0.65) {};
\node[draw, circle] (v6) at (0,2) {};
\draw (v5) edge (v6);
\node[draw, dashed, circle, minimum size=2.5cm] (v7) at (0,1) {};
\node[draw, circle, dotted] (v8) at (0,2.75) {};
\draw (v7) edge[dotted] (v8);
\end{tikzpicture}
\end{align*}
where $\hat{\Gamma}$ denotes the image of $\Gamma$ under the map $F_1$. Note that this notation includes an implicit summation over all possible ways to add a hair to a vertex of $\Gamma$ as defined in equation \ref{f1} for the map $F_1$.

In order to prove the following lemma we will need some properties of the proposed map $F_n$. 
First notice that 
\[
\begin{tikzpicture}[baseline=8ex]
\node[] (v2) at (0,0.3) {$\Gamma_0$};
\node[draw, dashed, circle, minimum size=0.9cm] (v3) at (0,0.3) {};
\node[draw, circle] (v4) at (0,1.25) {};
\draw (v3) edge (v4);
\node[draw, dashed, circle, minimum size=1.7cm] (v5) at (0,0.65) {};
\node[draw, circle] (v6) at (0,2) {};
\node[] at (-0.2,1.7) {1};
\node[] at (0.5,1.7) {2};
\draw (v5) edge (v6);
\draw (v6) edge[out=-30, in=70] (v5);
\node[draw, dashed, circle, minimum size=2.5cm] (v7) at (0,1) {};
\node[draw, circle, dotted] (v8) at (0,2.75) {};
\draw (v7) edge[dotted] (v8);
\end{tikzpicture}
=0.
\]
This follows from the fact that we have to sum over alle possible ways to connect edges $1$ and $2$ to a vertex to both the graph above, depicted as solid circle, as well as the graph below, depicted as dashed circle. Therefore, there is an odd symmetry of interchanging the two edges $1$ and $2$. Hence, the graph vanishes.

From this follows 
\begin{equation}
\Big[\stadpole,\hat{\Gamma}_1(\dots(\hat{\Gamma}_n(\Gamma_0)))\Big]=
\sum_{j=1}^n\hat{\Gamma}_1(\dots(\Big[\stadpole,\hat{\Gamma}_j\Big](\dots(\hat{\Gamma}_n(\Gamma_0)))))+\hat{\Gamma}_1(\dots(\hat{\Gamma}_n(\Big[\stadpole,\hat{\Gamma}_0\Big])))
\label{loop_dist}
\end{equation}
which is equivalent to
\begin{equation*}
\begin{tikzpicture}[baseline=8ex]
\node[] (v2) at (0,0.3) {$\Gamma_0$};
\node[draw, dashed, circle, minimum size=0.9cm] (v3) at (0,0.3) {};
\node[draw, circle] (v4) at (0,1.25) {};
\draw (v3) edge (v4);
\node[draw, dashed, circle, minimum size=1.7cm] (v5) at (0,0.65) {};
\node[draw, circle] (v6) at (0,2) {};
\draw (v5) edge (v6);
\node[draw, dashed, circle, minimum size=2.5cm] (v7) at (0,1) {};
\node[ minimum size=0.5cm] (v8) at (0,2.0) {};
\draw (v8) to [out=120, in=60, loop] (v8);
\end{tikzpicture}
=\sum_j
\begin{tikzpicture}[baseline=8ex]
\node[] (v2) at (0,0.3) {$\Gamma_0$};
\node[draw, dashed, circle, minimum size=0.9cm] (v3) at (0,0.3) {};
\node[draw, circle] (v4) at (0,1.45) {j};
\draw (v4) to [out=30, in=-30, loop] (v4);
\draw (v3) edge (v4);
\node[draw, dashed, circle, minimum size=2.4cm] (v5) at (0,1) {};
\node[draw, circle] (v6) at (0,2.75) {};
\draw (v5) edge (v6);
\node[ minimum size=0.5cm] (v8) at (0,2.0) {};
\end{tikzpicture}
+
\begin{tikzpicture}[baseline=8ex]
\node[] (v2) at (0,0.3) {$\Gamma_0$};
\draw (v2) to [out=120, in=60, loop] (v2);
\node[draw, dashed, circle, minimum size=1.7cm] (v5) at (0,0.65) {};
\node[draw, circle] (v6) at (0,2) {};
\draw (v5) edge (v6);
\node[draw, dashed, circle, minimum size=2.5cm] (v7) at (0,1) {};
\node[draw, circle] (v8) at (0,2.75) {};
\draw (v7) edge[] (v8);
\end{tikzpicture}.
\end{equation*}
Furthermore, we have
\[
\sum_{p,q;p+q=n}\sum_{\tau \in sh(p,q)}\hat{\Gamma}_{\tau(1)}(\dots(\hat{\Gamma}_{\tau(p)}(\Gamma_0')))\otimes\hat{\Gamma}_{\tau(p+1)}(\dots(\hat{\Gamma}_{\tau(n)}(\Gamma_0'')))
=\hat{\Gamma}_1(\dots(\hat{\Gamma}_n(\Gamma_0'\otimes\Gamma_0''))).
\]
From the last two properties we deduce that
\begin{equation}
\begin{aligned}
&\sum_{p,q;p+q=n}\sum_{\tau \in sh(p,q)}
\sedge\circ\Big(\hat{\Gamma}_{\tau(1)}(\dots(\hat{\Gamma}_{\tau(p)}(\Gamma_0'))),\hat{\Gamma}_{\tau(p+1)}(\dots(\hat{\Gamma}_{\tau(n)}(\Gamma_0'')))\Big)\\
&=\hat{\Gamma}_1(\dots(\hat{\Gamma}_n(
\begin{tikzpicture}[baseline=-0.5ex]
\node[](v) at(0,0) {$\Gamma_0'$};
\node[](w) at(1,0) {$\Gamma_0''$};
\draw (v) edge [out=60, in=120] (w);
\end{tikzpicture}
)))
\end{aligned}
\label{pull_edge_inside}
\end{equation}
which corresponds to
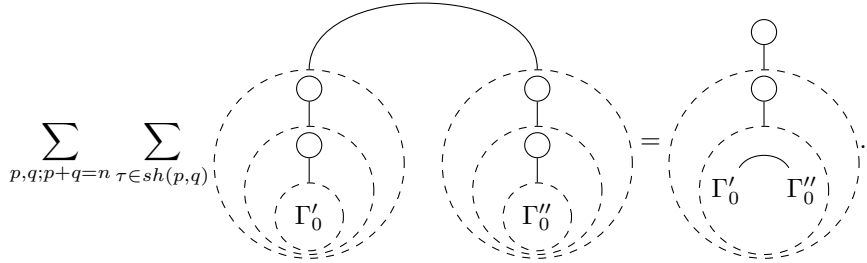
\begin{equation*}
\sum_{p,q;p+q=n}\sum_{\tau \in sh(p,q)}
\begin{tikzpicture}[baseline=8ex]
\node[] (v2) at (0,0.3) {$\Gamma_0'$};
\node[draw, dashed, circle, minimum size=0.9cm] (v3) at (0,0.3) {};
\node[draw, circle] (v4) at (0,1.25) {};
\draw (v3) edge (v4);
\node[draw, dashed, circle, minimum size=1.7cm] (v5) at (0,0.65) {};
\node[draw, circle] (v6) at (0,2) {};
\draw (v5) edge (v6);
\node[draw, dashed, circle, minimum size=2.5cm] (v7) at (0,1) {};
\node[] (v2) at (3,0.3) {$\Gamma_0''$};
\node[draw, dashed, circle, minimum size=0.9cm] (v3) at (3,0.3) {};
\node[draw, circle] (v4) at (3,1.25) {};
\draw (v3) edge (v4);
\node[draw, dashed, circle, minimum size=1.7cm] (v5) at (3,0.65) {};
\node[draw, circle] (v6) at (3,2) {};
\draw (v5) edge (v6);
\node[draw, dashed, circle, minimum size=2.5cm] (v8) at (3,1) {};
\draw (v7) edge[out=90, in=90] (v8);
\end{tikzpicture}
=
\begin{tikzpicture}[baseline=8ex]
\node[] (v1) at (-0.5,0.65) {$\Gamma_0'$};
\node[] (v2) at (0.5,0.65)  {$\Gamma_0''$};
\draw (v1) edge[out=60, in=120] (v2);
\node[draw, dashed, circle, minimum size=1.7cm] (v5) at (0,0.65) {};
\node[draw, circle] (v6) at (0,2) {};
\draw (v5) edge (v6);
\node[draw, dashed, circle, minimum size=2.5cm] (v7) at (0,1) {};
\node[draw, circle] (v8) at (0,2.75) {};
\draw (v7) edge[] (v8);
\end{tikzpicture}.
\end{equation*}
Finally we verify that 
\begin{equation}
\delta(\hat{\Gamma}_1(\dots(\hat{\Gamma}_n(\Gamma_0))))=\sum_{j=1}^n\hat{\Gamma}_1(\dots(\widehat{\tilde{\delta}\Gamma_j}(\dots(\hat{\Gamma}_n(\Gamma_0)))))
+ \hat{\Gamma}_1(\dots(\hat{\Gamma}_n(\delta\Gamma_0)))
\label{diff_dist}
\end{equation}
where the differential $\tilde\delta$ on the graph complex $GC_{2n}$ is defined in equation \ref{delta_tilde} and the differential $\delta$ on the graph complex $Graphs_{2n}$ is described in equation \ref{delta}.
Indeed we have 
\[
\Big[\intext,\hat{\Gamma}_1(\dots(\hat{\Gamma}_n(\Gamma_0)))\Big]=\sum_{j=1}^n\hat{\Gamma}_1(\dots(\Big[\intext,\hat{\Gamma}_j\Big](\dots(\hat{\Gamma}_n(\Gamma_0)))))+\hat{\Gamma}_1(\dots(\hat{\Gamma}_n(\Big[\intext,\Gamma_0\Big])))
\]
as well as
\[
(\hat{\Gamma}_1(\dots(\hat{\Gamma}_n(\Gamma_0))))\bullet\intint = \sum_{j=1}^n\hat{\Gamma}_1(\dots(\hat{\Gamma}_j\bullet\intint(\dots(\hat{\Gamma}_n(\Gamma_0)))))+\hat{\Gamma}_1(\dots(\hat{\Gamma}_n(\Gamma_0\bullet\intint)))
\]
and  
\[
\delta\hat{\Gamma}=\widehat{\tilde{\delta}\Gamma}.
\]
The last identity follows from equation \ref{diff_not_affect_hair}.

\subsection{A quasi-isomporhism}

Let us consider the graph complex
\[
\left(S^{+}\left(GC_{2n,conn}^{\geq2} \oplus \sedge\right)\left[[u]\right],\tilde{d}\right)
\]
where the differential $\tilde{d}$ is given by
\begin{equation}
\tilde{d}\Gamma:=\tilde{\delta}+u\frac{1}{2}\Big[\stadpole,\Gamma\Big]
\label{d_tilde}
\end{equation}
and $\tilde{\delta}$ refers to the differential on $GC_{2n}$ as defined in equation \ref{delta_tilde}. The differential $\tilde{d}$ acts as a coderivation on the deconcatenation copruduct.

Let $U$ and $V$ denote the cogenerators of the respective coalgebras as defined in equation \ref{cogenerators}. We show the following lemma:
\begin{lemma}
The map $F$ of coalgebras
\[
\left(S^{+}(U)\left[[u]\right],\tilde{d}\right)\xrightarrow[F]{}\left(S^{+}(V)\left[[u]\right],d\right)
\]
is a map of complexes, i.e. $F\circ\tilde{d}=d\circ F$. 
\label{map_complexes}
\end{lemma}

\begin{proof}
Let $\pi:S^+(V)\rightarrow V$ denote the projection onto the cogenerators. The projection of the differential $d$, defined in equation \ref{total_differential}, splits in two parts
\[
\pi\circ d = d^1 + d^2
\]
with 
\begin{align*}
d^1&: V\rightarrow V\\
d^2&:S^2(V)\rightarrow V.
\end{align*}
Thereby the part $d^2$ is given by
\begin{equation}
d^2(\Gamma_1\otimes\Gamma_2)=
u\left(\sedge\circ(\Gamma_1,\Gamma_2)+\dedge\circ(\Gamma_1,\Gamma_2)\right)=
u\left(
\begin{tikzpicture}[baseline=-0.5ex]
\node[] (g1) at (0,0) {$\Gamma_1$};
\node[] (g2) at (1,0) {$\Gamma_2$};
\draw (g1) edge [out=60, in=120] (g2);
\end{tikzpicture}
+
\begin{tikzpicture}[baseline=-0.5ex]
\node[] (g1) at (0,0) {$\Gamma_1$};
\node[] (g2) at (1,0) {$\Gamma_2$};
\draw (g1) edge [out=-60, in=-120, dashed] (g2);
\end{tikzpicture}
\right)
\label{d2}
\end{equation}

Projecting the relation 
\[
F\circ\tilde{d}=d\circ F
\]
to the cogenerators $V$ and restricting to $S^n(U)$ leads to the following two relations, which are to be shown:
\begin{equation}
\begin{aligned}
\sum_{j=1}^n F_n(\Gamma_1\dots \tilde{d}\Gamma_j\dots\Gamma_n)&=d^1F_n(\Gamma_1\dots\Gamma_n)\\
&+\sum_{p,q;p+q=n;p,q\neq0}\sum_{\tau \in sh(p,q)}d^2(F_p(\Gamma_{\tau(1)}\dots\Gamma_{\tau(p)})
\otimes F_q(\Gamma_{\tau(p+1)}\dots\Gamma_{\tau(n)}))
\label{eq_noexclass}
\end{aligned}
\end{equation}
\begin{equation}
\begin{aligned}
&\sum_{j=1}^{n-1} F_n(\sedge\Gamma_1\dots \tilde{d}\Gamma_j\dots\Gamma_{n-1})=d^1F_n(\sedge\Gamma_1\dots\Gamma_{n-1})\\
&+2\sum_{p,q;p+q=n-1;q\neq0}\sum_{\tau \in sh(p,q)}d^2(F_{p+1}(\sedge\Gamma_{\tau(1)}\dots\Gamma_{\tau(p)})\otimes F_q(\Gamma_{\tau(p+1)}\dots\Gamma_{\tau(n-1)}))
\label{eq_exclass}
\end{aligned}
\end{equation}

Let us first prove equation \ref{eq_noexclass}.
The term 
\[
d^1F_n(\Gamma_1\dots\Gamma_n) = A+B+C+D+E
\]
has the following contributions
\begin{align*}
A:=\delta F_n(\Gamma_1\dots\Gamma_n)&= \sum_{j=1}^n F_n(\Gamma_1\dots\tilde{\delta}\Gamma_j\dots\Gamma_n)\\
B:=u^{n-1}\Big[\sedge,\hat{\Gamma}_1(\dots(\hat{\Gamma}_n(\extnode)))\Big]
&=u^{n-1}\Big(2\sedge\circ_1(\hat{\Gamma}_1(\dots(\hat{\Gamma}_n(\extnode))))
-\hat{\Gamma}_1(\dots(\hat{\Gamma}_n(\sedge)))
\Big)\\
C:=u^{n-1}\Big[\dedge,\hat{\Gamma}_1(\dots(\hat{\Gamma}_n(\extnode)))\Big]&
=u^{n-1}\Big(2\dedge\circ_1(\hat{\Gamma}_1(\dots(\hat{\Gamma}_n(\extnode))))
-\hat{\Gamma}_1(\dots(\hat{\Gamma}_n(\dedge)))
\Big)\\
D:=u^{n}\Big[\stadpole,\hat{\Gamma}_1(\dots(\hat{\Gamma}_n(\extnode)))\Big]&=
u^{n}\sum_{j=1}^n\hat{\Gamma}_1(\dots(\Big[\stadpole,\hat{\Gamma}_j\Big](\dots(\hat{\Gamma}_n(\extnode)))))\\
E=u^{n}\Big[\dtadpole,\hat{\Gamma}_1(\dots(\hat{\Gamma}_n(\extnode)))\Big]&=0.
\end{align*}
In the calculation for A we used equation \ref{diff_dist}, whereas for D equation \ref{loop_dist} was applied.
The terms $A+D$ are equal to the left hand side of equation \ref{eq_noexclass}:
\[
\sum_{j=1}^n F_n(\Gamma_1\dots \tilde{d}\Gamma_j\dots\Gamma_n) = A + D.
\]

According to equation \ref{d2} the second term on the right hand side
\[
\sum_{p,q;p+q=n;p,q\neq0}\sum_{\tau \in sh(p,q)}d^2(F_p(\Gamma_{\tau(1)}\dots\Gamma_{\tau(p)})\otimes F_q(\Gamma_{\tau(p+1)}\dots\Gamma_{\tau(n)})) = F+G
\]
has the contributions
\begin{align*}
F:=u^{n-1}\sum_{p,q;p+q=n;p,q\neq0}\sum_{\tau \in sh(p,q)}
\sedge\Big(\hat{\Gamma}_{\tau(1)}(\dots(\hat{\Gamma}_{\tau(p)})),\hat{\Gamma}_{\tau(p+1)}(\dots(\hat{\Gamma}_{\tau(n)}))\Big)&= -B\\
G:=u^{n-1}\sum_{p,q;p+q=n;p,q\neq0}\sum_{\tau \in sh(p,q)}
\dedge\Big(\hat{\Gamma}_{\tau(1)}(\dots(\hat{\Gamma}_{\tau(p)})),\hat{\Gamma}_{\tau(p+1)}(\dots(\hat{\Gamma}_{\tau(n)}))\Big)&= -C
\end{align*}
which cancel the terms $B+C$ due to equation \ref{pull_edge_inside}.

Let us now verify equation \ref{eq_exclass}. The term 
\[
d^1F_n(\sedge\Gamma_1\dots\Gamma_{n-1}) = AA+BB+CC+DD+EE
\]
has the following contributions
\[
AA:=\delta F_n(\sedge\Gamma_1\dots\Gamma_{n-1})= \sum_{j=1}^{n-1} F_n(\sedge\Gamma_1\dots\tilde{\delta}\Gamma_j\dots\Gamma_{n-1})
\]
\[
\begin{aligned}
BB:=&u^{n-1}\Big[\sedge,\hat{\Gamma}_1(\dots(\hat{\Gamma}_{n-1}(\sedge)))+u\hat{\Gamma}_1(\dots(\hat{\Gamma}_{n-1}(\stadpole)))\Big]\\
=&u^{n-1}\Big(2\sedge\circ_1(\hat{\Gamma}_1(\dots(\hat{\Gamma}_{n-1}(\sedge))))
-4\hat{\Gamma}_1(\dots(\hat{\Gamma}_{n-1}(
\begin{tikzpicture}[baseline=-0.5ex]
\node[ext](v) at(0,0) {};
\node[ext](w) at(.5,0) {};
\node[ext](z) at(1,0) {};
\draw (v) edge [out=45, in=135] (w);
\draw (w) edge [out=45, in=135] (z);
\end{tikzpicture}
)))\Big)\\
&+u^n\Big(2\sedge\circ_1(\hat{\Gamma}_1(\dots(\hat{\Gamma}_{n-1}(\stadpole))))
-2\hat{\Gamma}_1(\dots(\hat{\Gamma}_{n-1}(
\begin{tikzpicture}[baseline=-0.5ex]
\node[ext](v) at(0,0) {};
\node[ext](w) at(.5,0) {};
\draw (v) edge [out=45, in=135] (w);
\draw (v) to [ out=120, in=60, loop] (v);
\end{tikzpicture}
)))\Big)
\end{aligned}
\]
\[
\begin{aligned}
CC:=&u^{n-1}\Big[\dedge,\hat{\Gamma}_1(\dots(\hat{\Gamma}_{n-1}(\sedge)))+u\hat{\Gamma}_1(\dots(\hat{\Gamma}_{n-1}(\stadpole)))\Big]\\
=&u^{n-1}\Big(2\dedge\circ_1(\hat{\Gamma}_1(\dots(\hat{\Gamma}_{n-1}(\sedge))))
-4\hat{\Gamma}_1(\dots(\hat{\Gamma}_{n-1}(
\begin{tikzpicture}[baseline=-0.5ex]
\node[ext](v) at(0,0) {};
\node[ext](w) at(.5,0) {};
\node[ext](z) at(1,0) {};
\draw (v) edge [out=45, in=135] (w);
\draw (w) edge [out=-45, in=-135, dashed] (z);
\end{tikzpicture}
)))\Big)\\
&+u^n\Big(2\dedge\circ_1(\hat{\Gamma}_1(\dots(\hat{\Gamma}_{n-1}(\stadpole))))
-2\hat{\Gamma}_1(\dots(\hat{\Gamma}_{n-1}(
\begin{tikzpicture}[baseline=-0.5ex]
\node[ext](v) at(0,0) {};
\node[ext](w) at(.5,0) {};
\draw (v) edge [out=-45, in=-135, dashed] (w);
\draw (v) to [ out=120, in=60, loop] (v);
\end{tikzpicture}
)))
-2\hat{\Gamma}_1(\dots(\hat{\Gamma}_{n-1}(
\begin{tikzpicture}[baseline=-0.5ex]
\node[ext](v) at(0,0) {};
\node[ext](w) at(.5,0) {};
\draw (v) edge [out=45, in=135] (w);
\draw (v) edge [out=-45, in=-135, dashed] (w);
\end{tikzpicture}
)))
\Big)\end{aligned}
\]
\[
\begin{aligned}
DD:&=u^{n}\Big[\stadpole,\hat{\Gamma}_1(\dots(\hat{\Gamma}_{n-1}(\sedge)))+u\hat{\Gamma}_1(\dots(\hat{\Gamma}_{n-1}(\stadpole)))\Big]\\
&=u^{n}\sum_{j=1}^{n-1}\Big(\hat{\Gamma}_1(\dots(\Big[\stadpole,\hat{\Gamma}_j\Big](\dots(\hat{\Gamma}_{n-1}(\sedge)))))+
\hat{\Gamma}_1(\dots(\Big[\stadpole,\hat{\Gamma}_j\Big](\dots(\hat{\Gamma}_{n-1}(\stadpole)))))\Big)
\end{aligned}
\]
\[
\begin{aligned}
EE:&=u^{n}\Big[\dtadpole,\hat{\Gamma}_1(\dots(\hat{\Gamma}_{n-1}(\sedge)))+u\hat{\Gamma}_1(\dots(\hat{\Gamma}_{n-1}(\stadpole)))\Big]\\
&=2u^{n}\hat{\Gamma}_1(\dots(\hat{\Gamma}_{n-1}(
\begin{tikzpicture}[baseline=-0.5ex]
\node[ext](v) at(0,0) {};
\node[ext](w) at(.5,0) {};
\draw (v) edge [out=45, in=135] (w);
\draw (v) edge [out=-45, in=-135, dashed] (w);
\end{tikzpicture}
))).
\end{aligned}
\]
Again in the calculation for AA we used equation \ref{diff_dist}, whereas for DD equation \ref{loop_dist} was applied.
The terms $AA+DD$ cancel again the left hand side of equation \ref{eq_exclass}:
\[
\sum_{j=1}^{n-1} F_n(\sedge\Gamma_1\dots \tilde{d}\Gamma_j\dots\Gamma_{n-1}) = AA + DD.
\]

Following equation \ref{d2} the second term on the right hand side
\[
2\sum_{p,q;p+q=n-1;q\neq0}\sum_{\tau \in sh(p,q)}d^2(F_{p+1}(\sedge\Gamma_{\tau(1)}\dots\Gamma_{\tau(p)})\otimes F_q(\Gamma_{\tau(p+1)}\dots\Gamma_{\tau(n-1)}))=FF+GG
\]
has the following contributions
\begin{align*}
FF:&=2u^{n-1}\sum_{p,q;p+q=n-1;q\neq0}\sum_{\tau \in sh(p,q)}
\sedge\Big(\hat{\Gamma}_{\tau(1)}(\dots(\hat{\Gamma}_{\tau(p)}(\sedge)))+u\hat{\Gamma}_{\tau(1)}(\dots(\hat{\Gamma}_{\tau(p)}(\stadpole))),\\
&\hat{\Gamma}_{\tau(p+1)}(\dots(\hat{\Gamma}_{\tau(n-1)}(\extnode)))\Big)\\
&=-BB
\end{align*}
\begin{align*}
GG:&=2u^{n-1}\sum_{p,q;p+q=n-1;q\neq0}\sum_{\tau \in sh(p,q)}
\dedge\Big(\hat{\Gamma}_{\tau(1)}(\dots(\hat{\Gamma}_{\tau(p)}(\sedge)))+u\hat{\Gamma}_{\tau(1)}(\dots(\hat{\Gamma}_{\tau(p)}(\stadpole))),\\
&\hat{\Gamma}_{\tau(p+1)}(\dots(\hat{\Gamma}_{\tau(n-1)}(\extnode)))\Big)\\
&=-(CC+EE)
\end{align*}
which cancel the terms $BB + CC +EE$ due to equation \ref{pull_edge_inside}.
\end{proof}

Using the fact that the map $F$ is a map of complexes we can prove the main result of this paper:
\begin{thm}
The map
\[
\left(S^{+}\left(GC_{2n,conn}^{\geq2} \oplus \sedge\right)\left[[u]\right],\tilde{d}\right)
\xrightarrow[F]{}
(Def,d)
\]
is a quasi-isomorphism.
\label{quasiiso}
\end{thm}
Here, $Def$ refers to the considered deformation complex in the form of equation \ref{def} and \ref{def_s+} and $d$ to the total differential defined in equation \ref{total_differential}. Finally, $\tilde{d}$ is defined in equation \ref{d_tilde}.

\begin{proof}
By lemma \ref{map_complexes} $F$ is a map of complexes.
Let us consider the bounded above complete descending filtration with respect to the power $p$ of $u$. The map $F$ is compatible with the filtration and the restriction to the associated graded vector spaces
\[
\left(u^pS^{+}\left(GC_{2n,conn}^{\geq2} \oplus \sedge\right),\tilde{\delta}\right)\rightarrow\left(gr^p Def, 
\delta + 
\Big[\dedge,\cdot \Big]
+
\Big[\sedge,\cdot\Big] \right)
\]
is a quasi-isomorphism by theorem 1.3 in \cite{Willwacher2015}.
Hence, the map F is a quasi-isomorphism, too.
\end{proof}

\subsection{Cohomology of the homotopy derivations of the $BV_{2n}$ operads}

From theorem \ref{quasiiso} we can deduce that the cohomology of the homotopy derivations of the $BV_{2n}$ operads are given by
\[
H(Der(BV_{2n,\infty})) = H\left(S^{+}\left(GC_{2n,conn}^{\geq2} \oplus \sedge\right)\left[[u]\right],\tilde{d}\right).
\]

Let us consider the tensor product over the ring $\mathbb{R}[[u]]$. We will denote the completed symmetric product space of a vector space $V$ with respect to the tensor product over $\mathbb{R}[[u]]$ by
\[
S^+_{\mathbb{R}[[u]]}(V).
\]
Note that the differential $\tilde{d}$ as defined in equation \ref{d_tilde} acts as coderivation on the coproduct of connected graphs. It does neither split connected graphs nor does it connect two connected components. Hence, we can apply the Künneth formula as well as the fact that taking comohology interchanges with taking invariants with respect to the symmetric group $\mathbb{S}_n$. Therefore, we can pull the cohomology inside the completed symmetric product. Since wie consider the completed symmetric product with respect to $\mathbb{R}[[u]]$ also the decoration with power series in $u$ can be interchanged with the product. Finally note that the extra class $\sedge$ is exact due to the odd symmetry of interchanging two edges but not closed under the differential $\tilde{d}$. Summarised we can write the homotopy derivations of the $BV_{2n}$ operads in the following form
\begin{equation}
\begin{aligned}
H(Der(BV_{2n,\infty})) &= H\left(S^{+}\left(GC_{2n,conn}^{\geq2} \oplus \sedge\right)\left[[u]\right],\tilde{d}\right)\\
&= S^{+}_{\mathbb{R}[[u]]}\left(\left(H\left(GC_{2n,conn}^{\geq2},\tilde{d}\right) \oplus \sedge \right)[[u]]\right).
\end{aligned}
\label{coho_hom_der_bv}
\end{equation}

\subsection{Zeroth cohomology isomorphic to $\mathfrak{grt}$}

As a corollary of theorem \ref{quasiiso} we can extend theorem 1.2 in \cite{Willwacher2015} by Willwacher and deduce that the cohomology of the homotopy derivations of the $BV_2$ operad is isomorphic to the Grothendieck-Teichm\"uller Lie algebra plus one class.
\begin{thm}
\[
H^0(Der(BV_{2,\infty}))\cong  \mathfrak{grt} := \mathfrak{grt_1}\rtimes\mathbb{R}
\]
where $\mathbb{R}$ acts on $\mathfrak{grt_1}$ by multiplication with the degree with respect to the grading on $\mathfrak{grt_1}$.
\label{zerocoho}
\end{thm}
\begin{proof}
Willwacher proved that the zeroth cohomology of the graph complex $GC_2$, considered as Lie algebra, is isomorphic to the Grothendieck-Teichm\"uller Lie algebra
\[
H^0(GC_{2,conn}) \cong \mathfrak{grt_1},
\]
\cite[Theorem 1.1]{Willwacher2015}, and deduced that 
\[
H^0(Der(e_{2,\infty}))\cong  \mathfrak{grt},
\]
\cite[Theorem 1.2]{Willwacher2015}.

Furthermore, Merkulov and Willwacher showed in \cite{MerkulovWillwacher2014} that 
\[
H^0\left(GC_{2,conn}[[u]],\tilde{d}\right) \cong \mathfrak{grt_1}
\]
where the differential $\tilde{d}$ is defined in equation \ref{d_tilde}. 

By proposition 3.4 in \cite{Willwacher2015} we have 
\[
H\left(GC^{\geq 2}_{2,conn}\right) = H(GC_{2,conn}) \oplus \bigoplus_{\stackanchor{$j\geq 3$}{ $j\equiv 3$ mod $4$}} \mathbb{R}[2-j].
\]
Here, the class $\mathbb{R}[2n-j]$ is represented by a loop with $j$ edges. The differential $\tilde{d}$ does neither split nor glue different connected components and oops are exact but not closed under it. Hence we also have
\[
H\left(GC^{\geq 2}_{2,conn},\tilde{d}\right) = H\left(GC_{2,conn},\tilde{d}\right) \oplus \bigoplus_{\stackanchor{$j\geq 3$}{ $j\equiv 3$ mod $4$}} \mathbb{R}[2-j].
\]

Note that the cohomological degree of a connected component corresponds to 2 (\#vertices - 1) - \#edges, and that $u$ is an even variable with degree 2.

The theorem follows form equation \ref{zerocoho} if we consider the cohomology in degree $0$.
\end{proof}

\section{Acknowledgement}
The Author would like to thank Prof. T. Willwacher for proposing the problem question, giving inputs and ideas as well as for his kind advice and support. Furthermore, the author was partially supported by the ERC grant GRAPHCPX (678156), awarded to Prof. T. Willwacher.

\printbibliography

\end{document}